\documentclass[a4paper,11pt]{amsart}

\usepackage[top=30truemm,bottom=25truemm,left=35truemm,right=35truemm]{geometry}

\usepackage{amsmath}
\usepackage{amssymb}
\usepackage{amsthm}
\usepackage{graphics}
\usepackage[abbrev,alphabetic]{amsrefs}
\usepackage{amscd}
\usepackage[dvipdfmx]{graphicx}
\usepackage{url}

\newcommand{\doi}[1]{\url{https://doi.org/#1}}

\usepackage[all,cmtip]{xy}
\usepackage{xypic}

\title[Shokurov's index conjecture for quotient singularities]
{Shokurov's index conjecture for quotient singularities}

\dedicatory{Dedicated to Professor Vyacheslav V.\ Shokurov, on the~occasion of his~seventieth~birthday}

\author{Yusuke Nakamura}
\address{Graduate School of Mathematical Sciences, 
the University of Tokyo, 3-8-1, Komaba, Meguro-ku, Tokyo, 153-8914, Japan.}
\email{nakamura@ms.u-tokyo.ac.jp}

\author{Kohsuke Shibata}
\address{Department of Mathematics, School of Engineering, 
Tokyo Denki University, Adachi-ku, Tokyo 120-8551, Japan.}
\email{shibata.kohsuke@mail.dendai.ac.jp}

\subjclass[2020]{Primary 14E18; Secondary 14E30, 14B05}

\keywords{minimal log discrepancy, Gorenstein index, Shokurov's index conjecture}

\newtheorem{thm}{Theorem}[section]

\newtheorem{prop}[thm]{Proposition}

\theoremstyle{definition}
\newtheorem{defi}[thm]{Definition}

\newtheorem{conj}[thm]{Conjecture}

\theoremstyle{remark}
\newtheorem{rmk}[thm]{Remark}

\newtheorem*{ackn}{Acknowledgements}

\begin{document}
\begin{abstract}
We prove Shokurov's index conjecture for quotient singularities. 
\end{abstract}

\maketitle

\section{Introduction}
Shokurov conjectured that the Gorenstein index of a 
$\mathbb{Q}$-Gorenstein germ can be bounded in terms of its dimension and minimal log discrepancy. 
\begin{conj}[Shokurov, cf.\ \cite{Kaw15}*{Question 5.2}]\label{conj:index}
For any $n \in \mathbb{Z}_{> 0}$ and $a \in \mathbb{R}_{\ge 0}$, there exists a positive integer $r(n,a)$ with the following condition. 
\begin{itemize}
\item If an $n$-dimensional $\mathbb{Q}$-Gorenstein variety $X$ and a closed point $p \in X$ satisfy 
$\operatorname{mld}_p(X) = a$, 
then the Cartier index of $K_X$ at $p$ is at most $r(n,a)$. 
\end{itemize}
\end{conj}

\noindent
Conjecture \ref{conj:index} may be useful to attack the ACC conjecture (see \cite{Sho04}) 
since the ACC conjecture is known for varieties with a fixed Gorenstein index 
and boundaries with fixed coefficients (\cite{Nak16b}). 
Conjecture \ref{conj:index} is known to be true in the following cases: 
\begin{enumerate}
\item[({\ref{conj:index}}.1)] 
When $(n,a) = (2,0)$, Conjecture \ref{conj:index} can be confirmed 
by the classification of surface singularities (cf.\ \cite{Sho93}*{Corollary 5.10}).

\item[({\ref{conj:index}}.2)] 
When $n = 2$, Conjecture \ref{conj:index} is proved by G.\ Chen and J.\ Han \cite{CH21}*{Theorem 1.4}. 
They proved Conjecture \ref{conj:index} for pairs (see \cite{CH21}*{Conjecture 6.3} for the generalized formulation). 

\item[({\ref{conj:index}}.3)] 
When $(n,a) = (3,0)$, Conjecture \ref{conj:index} is proved by Ishii in \cite{Ish00} and Fujino in \cite{Fuj01}. 
Fujino in \cite{Fuj01} also proved that Conjecture \ref{conj:index} for $a=0$ and arbitrary $n$ assuming 
another conjecture on some boundedness of the birational automorphisms (see \cite{Fuj01}*{Theorem 0.2} for the detail). 

\item[({\ref{conj:index}}.4)] 
When the considered $X$ is a terminal threefold, 
Conjecture \ref{conj:index} follows from the classification by Kawamata \cite{Kaw92}. 
J.\ Han, J.\ Liu, and Y.\ Luo in \cite{HLL}*{Theorem 1.5} proved Conjecture \ref{conj:index} for three-dimensional terminal pairs. 

\item[({\ref{conj:index}}.5)] 
When the considered $X$ is a canonical threefold, 
Conjecture \ref{conj:index} is proved by Kawakita \cite{Kaw15}. 

\item[({\ref{conj:index}}.6)]
When the considered $X$ has toric singularities, 
Conjecture \ref{conj:index} is proved by Ambro \cite{Amb09} (cf.\ \cite{Amb16}*{Section 5.4}). 

\item[({\ref{conj:index}}.7)]
When the considered $X$ has only quotient singularity, Moraga \cite{Mor22} shows the existence of $\epsilon _n \in \mathbb{R}_{>0}$ (for each $n$) 
such that Conjecture \ref{conj:index} holds for $(n,a) = (n, \epsilon _n)$.
\end{enumerate}

The following theorem is our main theorem. 

\begin{thm}[$=$\ Theorem \ref{thm:quotp}]\label{thm:quot}
Conjecture \ref{conj:index} holds for quotient singularities. 
\end{thm}

The main idea of the proof of Theorem \ref{thm:quot} is to use Jordan's theorem (Theorem \ref{thm:Jordan}) 
on finite linear groups and reduce it to the abelian quotient case,
where Conjecture \ref{conj:index} is proved by Ambro (see ({\ref{conj:index}}.6) above). 
This idea is due to Moraga \cite{Mor22}. 

In Section \ref{section:GC}, we prove Shokurov's Gorensteiness conjecture for quotient singularities. 
This conjecture states that we can take $r(n, n-1) = 1$ in Conjecture \ref{conj:index}.

\begin{ackn}
We would like to thank Professor Vyacheslav V.\ Shokurov for 
his kind comments and suggestions to our preprint. 
The first author is partially supported by JSPS KAKENHI No.\ 18K13384, 22K13888, 18H01108, and JPJSBP120219935. 
The second author is partially supported by JSPS KAKENHI No.\ 19K14496.
\end{ackn}

\section*{Notation}

\begin{itemize}
\item 
We basically follow the standard notions and terminologies in \cite{Har77} and \cite{Kol13}. 

\item 
Throughout this paper, $k$ is an algebraically closed field of characteristic zero. 
We say that $X$ is a \textit{variety} if 
$X$ is an integral scheme that is separated and of finite type over $k$. 
\end{itemize}

\section{Preriminaries}

First, we clarify the definition of quotient singularities in this paper. 
For a local ring $A$, let $\widehat{A}$ denote the completion of $A$ at the maximal ideal. 
\begin{defi}\label{defi:quot}
Let $X$ be a variety of dimension $n$ and let $p \in X$ be a closed point. 
We say that $X$ has a \textit{quotient singularity} at $p$ if 
there exists a finite subgroup $G \leqslant \operatorname{GL}_n(k)$ such that 
$\widehat{\mathcal{O}}_{X,p} \simeq \widehat{\mathcal{O}}_{M,x_0}$, where 
$M = \mathbb{A}^n _k/G$ is the quotient variety obtained by the natural linear action of $G$ on $\mathbb{A}^n_k$, 
and $x_0 \in M$ is the image of the origin of $\mathbb{A}^n _k$. 
\end{defi}

Next, we recall the definition of the Gorenstein indices and the minimal log discrepancies for $\mathbb{Q}$-Gorenstein normal singularities. 
\begin{defi}\label{defi:mld}
Let $X$ be a normal variety and let $p \in X$ be a closed point. 
Suppose that $X$ has a $\mathbb{Q}$-Gorenstein singularity at $p$ (i.e.\ $r K_X$ is Cartier at $p$ for some positive integer $r$). 
\begin{enumerate}
\item
The \textit{Gorenstein index} of $X$ at $p$ is the smallest positive integer $r$ such that $rK_X$ is Cartier at $p$. 

\item
Let $f: X' \to X$ be a proper birational morphism from a normal variety $X'$ and let $E$ be a prime divisor on $X'$. 
Then the \textit{log discrepancy} $a_E(X)$ is defined by
\[
a_E(X) := 1 + \operatorname{ord}_E(K_{X'/X}), 
\]
where $K_{X'/X} := K_{X'} - f^* K_X$ is the relative canonical divisor. 
We define
the \textit{minimal log discrepancy} $\operatorname{mld}_p(X)$ at $p$ by
\[
\operatorname{mld}_p(X) := \inf_{c_X(E) = \{ p \}} a_E(X)
\]
if $\dim X \ge 2$, where the infimum is taken over all prime divisors $E$ over $X$ with center $c_X(E) = \{ p \}$. 
When $\dim X = 1$, we define $\operatorname{mld}_p(X) := \inf_{c_X(E) = \{ p \}} a_E(X)$ if $\inf_{c_X(E) = \{ p \}} a_E(X) \ge 0$, and 
we define $\operatorname{mld}_p(X) := - \infty$ otherwise. 
It is known that $\operatorname{mld}_p(X) \in \mathbb{Q}_{\ge 0} \cup \{ - \infty \}$ in general (cf.\ \cite{KM98}*{Corollary 2.31}). 
\end{enumerate}
\end{defi}

For a finite group $G \leqslant \operatorname{GL}_{n}(k)$ and $g \in G$, we define $d(G)$ and $\operatorname{age}'(g)$ as follows. 
\begin{defi}\label{defi:d}
Let $n$ be a positive integer and 
let $G \leqslant \operatorname{GL}_{n}(k)$ be a finite subgroup. 
Let $d := \# G$ be the order of $G$, and let $\xi \in k$ be a primitive $d$-th root of unity. 
\begin{enumerate}
\item 
We define a positive integer $d(G)$ by
\[
d(G) := \min \bigl\{ \ell \in \mathbb{Z}_{>0} \ \big| \  \text{$(\det (g))^{\ell} = 1$ holds for any $g \in G$} \bigr\}. 
\]

\item 
Let $g \in G$. 
Since $g$ has finite order, $g$ is conjugate to a diagonal matrix $\operatorname{diag}(\xi ^{e_1}, \ldots , \xi ^{e_n})$ 
with $1 \le e_i \le d$. 
Then, we define $\operatorname{age}'(g) := \sum _{i = 1} ^n \frac{e_i}{d}$. 
\end{enumerate}
\end{defi}

\begin{rmk}\label{rmk:age}
The notation $\operatorname{age}'(g)$ defined in Definition \ref{defi:d}(2) is not usually used outside of this paper. 
A prime ${}'$ is added to distinguish it from the usual age (denoted by $\operatorname{age}(g)$) defined in some literatures (cf.\ \cite{Kol13}*{Definition 3.20}). 
Note that we have the relation $\operatorname{age}(g) = \operatorname{age}'(g) - \# \{1 \le i \le n \mid e_i = d \}$. 
\end{rmk}

\begin{rmk}\label{rmk:p_ref}
In Theorem \ref{thm:d} and Proposition \ref{prop:age} below, we assume that $G$ does not contain a pseudo-reflection. 
In this remark, we will explain that this assumption is not harmful in applications. 

Let $n$ be a positive integer and 
let $G \leqslant \operatorname{GL}_{n}(k)$ be a finite subgroup. 
Let $G_{\rm pr} \leqslant G$ be the subgroup generated by the pseudo-reflections in $G$. 
Then, we have $\mathbb{A}^n_k / G \simeq \mathbb{A}^n_k / (G/G_{\rm pr})$ by the Chevalley-Shephard-Todd Theorem (cf.\ \cite{Kol13}*{Subsection 3.18}). 
Therefore, when considering a quotient variety $\mathbb{A}^n_k / G$, we may assume that $G$ contains no pseudo-reflection. 
\end{rmk}

In \cite{Wes91}, Weston  proved that the Gorenstein index of $\mathbb{A}^n_k /G$ at 
the image $x_0$ of the origin of $\mathbb{A}^n_k$ coincides with $d(G)$ when $G$ 
has no pseudo-reflections. 
\begin{thm}[\cite{Wes91}*{Theorem 2.2}]\label{thm:d}
Let $n$ be a positive integer and 
let $G \leqslant \operatorname{GL}_{n}(k)$ be a finite subgroup. 
Suppose that $G$ does not contain a pseudo-reflection. 
Let $x_0 \in \mathbb{A}^n_k / G$ be the image of the origin of $\mathbb{A}^n _k$. 
Then, $d(G)$ coincides with the Gorenstein index of $\mathbb{A}^n_k / G$ at $x_0$. 
\end{thm}

\begin{rmk}
Weston's theorem \cite{Wes91}*{Theorem 2.2} treats a general setting, 
where the finite group $G$ may contain a pseudo-reflection. 
Note that in the notation in \cite{Wes91}*{Hypotheses 2.1}, 
we have $d_{ij} = 1$ and $g'_j = g_j$ for any $i$ and $j$ when $G$ does not contain a pseudo-reflection. 
\end{rmk}

The minimal log discrepancy of $\mathbb{A}^n_k /G$ at $x_0$ can be described by $\operatorname{age}'(g)$. 

\begin{prop}\label{prop:age}
Let $n$, $G$ and $x_0$ be as in Theorem \ref{thm:d}. 
Then it follows that 
\[
\operatorname{mld}_{x_0} \bigl( \mathbb{A}^n_k / G \bigr) = \min \{ \operatorname{age}'(g) \mid g \in G \}.
\]
\end{prop}
\begin{proof}
First, note that this formula is well-known when $G$ is an abelian group (see Remark \ref{rmk:age1}). 

For $\gamma \in G$, let $\langle \gamma \rangle \leqslant G$ denote the subgroup generated by $\gamma$, and 
let $x_{\gamma} \in \mathbb{A}^n_k / \langle \gamma \rangle$ denote the image of the origin of $\mathbb{A}^n_k$. 
Then, by \cite{NS22}*{Corollary 4.12} (see also \cite{NS22}*{Remark 4.13}), we have 
\[
\operatorname{mld}_{x_0} \bigl( \mathbb{A}^n_k / G \bigr) = 
\min_{\gamma \in G} \operatorname{mld}_{x_{\gamma}} \bigl( \mathbb{A}^n_k / \langle \gamma \rangle \bigr). 
\]
Therefore, the assertion follows from the cyclic quotient case. 
\end{proof}

\begin{rmk}\label{rmk:age1}
When $G$ is an abelian group, Proposition \ref{prop:age} seems to be well-known to experts. 
However, since we could not find a good reference, we give a proof below for the reader's convenience. 
We refer to \cite{CLS} for the standard notions on toric varieties. 

Suppose that $G$ in Proposition \ref{prop:age} is an abelian group. 
Since $G$ is a finite abelian group, its elements are simultaneously diagonalizable. 
Hence, we may assume that $G$ consists of diagonal matrices from the beginning. 
Let $d \in \mathbb{Z}_{>0}$ and $\xi \in k$ be as in Definition \ref{defi:d}. 
Each $g \in G$ can be written as $g = \operatorname{diag}\bigl( \xi^{e_{1}^{(g)}}, \ldots, \xi^{e_{n}^{(g)}} \bigr)$. 
We define $u^{(g)} := \bigl( e_{1}^{(g)}/d, \ldots, e_{n}^{(g)}/d \bigr) \in \mathbb{R}^n$, and 
define a lattice $N \subset \mathbb{R}^n$ by 
$N := \mathbb{Z}^n + \sum _{g \in G} \mathbb{Z} u^{(g)}$. 
Let $u_1, \ldots , u_n \in \mathbb{R}^n$ be the standard basis 
(i.e.\ $u_i$ is the vector with $1$ in the $i$-th coordinate and $0$'s elsewhere). 
Let $\sigma \subset \mathbb{R}^n = N_{\mathbb{R}}$ be the cone spanned by $u_1, \ldots , u_n$. 
Let $U_{\sigma, N}$ denote the toric variety corresponding to $\sigma \subset N_{\mathbb{R}}$. 
Then we have $U_{\sigma, N} \simeq \mathbb{A}^n_k / G$ (cf.\ \cite{CLS}*{Example 1.3.20}). 

Since $G$ does not contain a pseudo-reflection, $u_1, \ldots , u_n \in N$ are primitive vectors 
(i.e.\ there are no integer $\ell \ge 2$ and $u \in N$ satisfying $u_i = \ell u$). 
Therefore, the support function $\varphi: \sigma \to \mathbb{R}$ of the $\mathbb{Q}$-Cartier divisor $K_{U_{\sigma, N}}$ 
is given by the function satisfying $\varphi((a_1, \ldots, a_n)) = \sum _i a_i$ for any $(a_1, \ldots , a_n) \in \sigma \subset \mathbb{R}^n$. 
Furthermore, 
we have the one-to-one correspondence between the primitive vectors $u$ in $\operatorname{int}(\sigma) \cap N$ and 
the torus invariant divisors $D_u$ over $U_{\sigma, N}$ whose centers are the torus invariant closed point. 
Then, we have $a_{D_u}(U_{\sigma, N}) = \varphi(u)$ by \cite{CLS}*{Lemma 11.4.10}. 
Therefore, we get 
\[
\operatorname{mld}_{x_0} \bigl( \mathbb{A}^n_k / G \bigr) 
= \min _{u \in \operatorname{int}(\sigma) \cap N} \varphi(u)
= \min _{u \in (0,1]^n \cap N} \varphi(u),
\]
and the right hand side is equal to $\min \{ \operatorname{age}'(g) \mid g \in G \}$ since 
we have $\varphi(u^{(g)}) = \operatorname{age}'(g)$ for $g \in G$. 
\end{rmk}

\begin{rmk}\label{rmk:age2}
In the same setting as in Proposition \ref{prop:age}, it is known that 
\[
\operatorname{mld} \bigl( \mathbb{A}^n_k / G \bigr) = 
\min \{ \operatorname{age}(g) \mid g \in G \setminus \{ 1_G \} \}, 
\]
where $\operatorname{mld}(X) := \inf _{E} a_E(X)$ is the total log discrepancy 
(see Remark \ref{rmk:age} for the difference between $\operatorname{age}'(g)$ and $\operatorname{age}(g)$). 
This formula is proved by Yasuda \cite{Yas06} using the motivic integration method. 
Proposition \ref{prop:age} can also be directly proved in a similar way without reducing the abelian quotient case. 
Indeed, it follows from \cite{NS22}*{Theorem 4.8} or \cite{NS2}*{Theorem 7.9}. 
\end{rmk}

\section{Shokurov's index conjecture for quotient singularities}

In this section, we prove Shokurov's index conjecture for quotient singularities. 
First, note that Ambro in \cite{Amb09} proved 
this conjecture for toric varieties, in particular for abelian quotient singularities. 
\begin{thm}[Ambro \cite{Amb09}]\label{thm:abel}
Conjecture \ref{conj:index} is true for toric singularities, 
in particular for abelian quotient singularities.
\end{thm}

We use Jordan's theorem on finite linear groups. 
\begin{thm}[Jordan's theorem on finite linear groups \cite{Jor78}]\label{thm:Jordan}
For a positive integer $n$, 
there exists a positive integer $c_n$ with the following condition: 
\begin{itemize}
\item 
Suppose that $K$ is a field of characteristic zero and that 
$G \leqslant \operatorname{GL}_n(K)$ is a finite subgroup. 
Then, there exists an abelian normal subgroup $H \leqslant G$ 
such that its index satisfies $[G:H] \le c_n$. 
\end{itemize}
\end{thm}

\begin{rmk}
Although Jordan's theorem is usually asserted for $K = \mathbb{C}$, 
it holds for any field $K$ of characteristic zero with the same bound $c_n$. 
Since $G$ is a finite subgroup, $G \leqslant \operatorname{GL}_n(K')$ holds for some subfield $K' \subset K$ with $\# K' = \aleph _0$. 
Since such $K'$ can be embedded into $\mathbb{C}$, the assertion follows from the case where $K = \mathbb{C}$. 
\end{rmk}

\begin{thm}\label{thm:quotp}
Conjecture \ref{conj:index} is true for quotient singularities. 
\end{thm}
\begin{proof}
Let $n \in \mathbb{Z}_{>0}$ and $a \in \mathbb{R}_{\ge 0}$. 
Suppose that a variety $X$ of dimension $n$ has a quotient singularity at a closed point $p \in X$. 
Furthermore, assume $\operatorname{mld}_p(X) = a$. 
By the definition of quotient singularities (Definition \ref{defi:quot}), 
there exists a finite subgroup $G \leqslant \operatorname{GL}_n(k)$ such that 
$\widehat{\mathcal{O}}_{X,p} \simeq \widehat{\mathcal{O}}_{M,x_0}$, where 
$M = \mathbb{A}^n _k/G$ is the quotient variety and $x_0 \in M$ is the image of the origin of $\mathbb{A}^n _k$. 
Let $G_{\rm pr} \leqslant G$ be the subgroup generated by the pseudo-reflections in $G$. 
Then, by replacing $G$ with $G/G_{\rm pr}$, we may assume that $G$ does not contain a pseudo-reflection (cf.\ \cite{Kol13}*{3.18}).
Note that the minimal log discrepancy is preserved under  completion (cf.\ \cite{NS2}*{Remark 3.1}). 
Furthermore, the Gorenstein index is also preserved under  completion (cf.\ \cite{dFEM11}*{Proposition A.14}). 
Therefore, we may assume $X = M = \mathbb{A}^n _k/G$ and $p = x_0$ from the beginning. 

By Jordan's theorem (Theorem \ref{thm:Jordan}), we can take an abelian normal subgroup $H \leqslant G$ such that $[G:H] \le c_n$. 
By Proposition \ref{prop:age}, we can take $g \in G$ and $h \in H$ such that 
\[
a = \operatorname{mld}_{x_0} \bigl( \mathbb{A}^n _k/G \bigr) = \operatorname{age}'(g), \quad
\operatorname{mld}_{x_0} \bigl( \mathbb{A}^n _k/H \bigr) = \operatorname{age}'(h). 
\]
We set $c' := c_n !$ and 
\[
h' := h \cdot (g h g^{-1}) \cdot (g^2 h g^{-2}) \cdots (g^{c' - 1} h g^{- (c' - 1)}). 
\]
Since $H$ is a normal subgroup of $G$, we have $g^i h g^{-i} \in H$ for any $i \ge 0$. 
Furthermore, since $[G:H] \le c_n$, we have $g^{c'} \in H$, and hence $g^{c'} h g^{- c'} = g^{c'} g^{- c'} h  = h$. 
Therefore, we have 
\[
g h' g^{-1} = (g h g^{-1}) \cdot (g^2 h g^{-2}) \cdots (g^{c' - 1} h g^{- (c' - 1)}) \cdot h = h'. 
\]

Let $\langle g, h' \rangle \leqslant G$ denote the subgroup generated by $g$ and $h'$. 
Since we have 
\[
\operatorname{age}'(g) = \operatorname{mld}_{x_0} \bigl( \mathbb{A}^n _k/G \bigr) 
= \min _{\alpha \in G} \operatorname{age}'(\alpha) 
\le \min _{\alpha \in \langle g, h' \rangle} \operatorname{age}'(\alpha) 
\le \operatorname{age}'(g), 
\]
it follows that 
\[
\operatorname{mld}_{x_0} \bigl( \mathbb{A}^n _k/\langle g, h' \rangle \bigr)
= \min _{\alpha \in \langle g, h' \rangle} \operatorname{age}'(\alpha)
= \operatorname{age}'(g) 
= a. 
\]
Since $\langle g, h' \rangle$ is an abelian group, by Theorem \ref{thm:abel}, 
we can conclude that $d(\langle g, h' \rangle)$ divides $r'_1(n, a)$ for some positive integer $r'_1(n,a)$ depending only on $n$ and $a$. 
Since $\det (h') = (\det (h))^{c'}$, the order of $\det (h)$ divides $c_n ! \cdot r'_1(n, a)$. 
In particular, $\operatorname{mld}_{x_0} \bigl( \mathbb{A}^n _k/H \bigr) = \operatorname{age}'(h)$ 
is contained in the finite set $\frac{1}{c_n ! \cdot r'_1(n, a)} \mathbb{Z} \cap (0,n]$. 
By applying Theorem \ref{thm:abel} to $H$, we can conclude that 
$d(H)$ divides $r'_2(n, a)$ for some positive integer $r'_2(n,a)$ depending only on $n$ and $a$. 
Since $[G:H] \le c_n$, we have $\alpha ^{c_n !} \in H$ for any $\alpha \in G$. 
Therefore, $d(G)$ divides $c_n ! \cdot d(H)$, and in particular, $d(G)$ divides $c_n ! \cdot r'_2(n,a)$. 
We complete the proof. 
\end{proof}

\begin{rmk}
As a corollary of Theorem \ref{thm:quotp}, 
we could prove Conjecture \ref{conj:index} for surface singularities. 
Note that log terminal surface singularities are quotient singularities (cf.\ \cite{Kaw84}*{Corollary 1.9}). 
More generally, G.\ Chen and J.\ Han in \cite{CH21}*{Theorem 1.4} proved that 
Shokurov's index conjecture for two-dimensional pairs (\cite{CH21}*{Conjecture 6.3}). 
\end{rmk}

\section{Shokurov's Gorensteinness conjecture for quotient singularities}\label{section:GC}

In this section, 
we prove Shokurov's Gorensteinness conjecture (Conjecture \ref{conj:bound}(3)) for quotient singularities. 
This conjecture states that we can take $r(n, n-1) = 1$ in Conjecture \ref{conj:index}. 

\begin{conj}[Shokurov \cite{Sho88}*{Propblem 5}]\label{conj:bound}
Let $X$ be an $n$-dimensional $\mathbb{Q}$-Gorenstein variety, 
and let $p \in X$ be a closed point.
Then the following assertions hold. 
\begin{enumerate}
\item 
We have $\operatorname{mld}_p(X) \le n$. 

\item
If $\operatorname{mld}_p(X) > n-1$, then $X$ is smooth at $p$. 

\item
If $\operatorname{mld}_p(X) = n-1$, then $X$ is Gorenstein at $p$.
\end{enumerate}
\end{conj}

\begin{rmk}
In \cite{Amb99}, Ambro proved Conjecture \ref{conj:bound}(1) and (2) for toric singularities 
(see Theorem 4.1 and Proposition 4.2 in \cite{Amb99}). 
\end{rmk}

\begin{thm}\label{thm:GC}
Conjecture \ref{conj:bound} is true when $X$ has a quotient singularity at $p$. 
\end{thm}
\begin{proof}
Suppose that $X$ has a quotient singularity at $p$. 
In the same way as in the proof of Theorem \ref{thm:quotp}, 
we may assume that $X$ is of the form $X  = \mathbb{A}^n _k/G$, 
where $G$ is a finite subgroup of $\operatorname{GL}_n(k)$, 
and $p$ is the image $x_0$ of the origin of $\mathbb{A}^n _k$. 
We may also assume that $G$ does not contain pseudo-reflections. 

Conjecture \ref{conj:bound}(1) easily follows from Proposition \ref{prop:age}. 

Suppose that $X$ is singular at $p$. 
Then, it follows that $G \not = \{ 1 _G \}$ since $G$ does not contain pseudo-reflections.  
Take any $g \in G \setminus \{ 1_G \}$. 
Let $d$ be the order of $g$, and let $\xi \in k$ be a primitive $d$-th root of unity. 
Suppose that $g$ is conjugate to the diagonal matrix 
$\operatorname{diag}(\xi ^{e_1}, \ldots , \xi ^{e_n})$ with $1 \le e_i \le d$. 
We may assume that $e_1 \le e_2 \le \cdots \le e_n$. 
Let $\ell$ be the maximum integer satisfying $e_{\ell} < d$. 
Then, we have 
\begin{align*}
\operatorname{age}'(g) &= \frac{e_1}{d} + \cdots + \frac{e_{\ell}}{d} + n - \ell, \\
\operatorname{age}'(g^{-1}) &= \left( 1 - \frac{e_1}{d} \right) + \cdots + \left( 1 - \frac{e_{\ell}}{d} \right) + n - \ell. 
\end{align*}
Note that we have $\ell \ge 2$ since $g$ is not a pseudo-reflection. 
By Proposition \ref{prop:age}, we have 
\[
\operatorname{mld}_p(X) 
\le \min \bigl\{ \operatorname{age}'(g), \operatorname{age}'(g^{-1}) \bigr\}
\le \frac{1}{2} \bigl( \operatorname{age}'(g) + \operatorname{age}'(g^{-1}) \bigr)
= n - \frac{\ell}{2} \le n -1. 
\]
Therefore, Conjecture \ref{conj:bound}(2) is proved. 

Suppose $\operatorname{mld}_p(X) = n-1$. 
Then, by the inequalities above, 
it follows that $\operatorname{age}'(g) = n-1$ for any $g \in G \setminus \{ 1_G \}$. 
Therefore, $K_X$ is Cartier at $p$ by Theorem \ref{thm:d}. 
Since quotient singularities are always Cohen-Macaulay, 
$X$ is Gorenstein at $p$. Hence, Conjecture \ref{conj:bound}(3) is proved. 
\end{proof}

\begin{thm}
Conjecture \ref{conj:bound}(3) is true for toric singularities. 
\end{thm}
\begin{proof}
Let $M := \mathbb{Z}^n$, and let $N := \operatorname{Hom}_{\mathbb{Z}}(M, \mathbb{Z})$ be its dual lattice. 
Let $X = U_{\sigma}$ be an $n$-dimensional $\mathbb{Q}$-Gorenstein affine toric variety 
corresponding to a strongly rational polyhedral cone $\sigma \subset N_\mathbb{R}$. 
Let $p \in U_{\sigma}$ be a closed point. 
For a face $\sigma '$ of $\sigma$, $O(\sigma')$ denotes the torus orbit corresponding to $\sigma'$ (cf.\ \cite{CLS}*{Theorem 3.2.6}).  
Let $\sigma _p$ be the face of $\sigma$ corresponding to the torus orbit $T_N \cdot p$ of $p$ (i.e.\ $T_N \cdot p = O(\sigma _p)$). 
Then, $U_{\sigma _p}$ is an open subset of $U_{\sigma}$ containing $p$.  
Therefore, replacing $\sigma$ with $\sigma _p$, we may assume that $T_N \cdot p = O(\sigma)$ from the beginning. 

Suppose that $\sigma$ is generated by primitive vectors $u_1, \ldots , u_{\ell} \in N$. 
Since $U_{\sigma}$ is $\mathbb{Q}$-Gorenstein, the support function $\varphi : \sigma \to \mathbb{R}$ of $K_{U_{\sigma}}$ is given by 
$\varphi (u) = \langle m, u \rangle$ for some $m \in M_\mathbb{Q}$. 

Let $c := \dim \sigma$, and let $\tau \subset N_\mathbb{R}$ be the cone generated by $u_1, \ldots, u_c \in N$. 
Since $\tau$ is a simplicial cone, 
the corresponding toric variety $U_{\tau}$ has abelian finite quotient singularities (cf.\ \cite{CLS}*{Theorem 11.4.8}). 
Furthermore, the same $m \in M_{\mathbb{Q}}$ gives the support function of $K_{U_{\tau}}$. 
Since $T_N \cdot p = O(\sigma)$, we have 
\[
\operatorname{mld}_{p} (U_{\sigma} ) 
= \dim O(\sigma) + \operatorname{mld}_{\eta _{O(\sigma)}} ( U_{\sigma})
= n -c + \min _{u \in \operatorname{int}(\sigma) \cap N} \langle m, u \rangle
\]
by \cite{Amb99}*{Proposition 2.1}. 
Here, $\operatorname{mld}_{\eta _{O(\sigma)}} ( U_{\sigma} )$ is defined as $\inf_{c_{U_{\sigma}}(E) = \overline{O(\sigma)}} a_E(X)$, 
where the infimum is taken over all prime divisors $E$ over $U_{\sigma}$ with center $c_{U_{\sigma}}(E) = \overline{O(\sigma)}$.
Similarly, for any $x_{\tau} \in O(\tau) \subset U_{\tau}$, we have 
\[
\operatorname{mld}_{x_\tau} ( U_{\tau} )
= n - c  + \min _{u \in \operatorname{int}(\tau) \cap N} \langle m, u \rangle. 
\] 
Since $\operatorname{int}(\tau)\subset \operatorname{int}(\sigma)$, we conclude
\[
\operatorname{mld}_{p} ( U_{\sigma} ) \le \operatorname{mld}_{x_\tau} ( U_{\tau} ). 
\]

Suppose $\operatorname{mld}_{p} ( U_{\sigma} )= n-1$. 
Then, we have $\operatorname{mld}_{x_{\tau}} ( U_{\tau}) \ge n-1$ for any $x_{\tau} \in O(\tau)$. 
Therefore, by Theorem \ref{thm:GC}, $K_{U_{\tau}}$ is Cartier at any $x_{\tau} \in O(\tau)$. 
Since $U_{\tau}$ is the only torus invariant open subset of $U_{\tau}$ containing $O(\tau)$, we conclude that 
$K_{U_{\tau}}$ is a Cartier divisor on $U_{\tau}$. 
Hence, there exists $m' \in M$ such that $m'$ gives the support function of $K_{U_{\tau}}$. 
Since $m - m' \in \tau ^{\perp} \cap M _{\mathbb{Q}} = \sigma ^{\perp} \cap M _{\mathbb{Q}}$ (cf.\ \cite{CLS}*{Theorem 4.2.8}), 
$m'$ also gives the support function of $K_{U_{\sigma}}$. 
This proves that $K_{U_{\sigma}}$ is also a Cartier divisor. 
Since toric varieties are always Cohen-Macaulay, 
we conclude that $U_{\sigma}$ is also Gorenstein. 
\end{proof}


\end{document}